\documentclass[11pt]{amsart}

\usepackage{amsmath}
\usepackage{amssymb}
\usepackage{graphicx}
\usepackage[
  citestyle=numeric,
  bibstyle=numeric
]{biblatex}

\newtheorem{thm}{Theorem}[section]
\newtheorem*{thm*}{Theorem}
\newtheorem{prop}[thm]{Proposition}

\newtheorem*{lem*}{Lemma}
\newtheorem*{cor*}{Corollary}
\theoremstyle{definition}
\newtheorem{def_n}[thm]{Definition}
\newtheorem{cor}{Corollary}[section]

\newtheorem{rem}[thm]{Remark}

\numberwithin{equation}{section}

\newcommand{\Hcal}{\mathcal{H}} % define \Hcal here

\title{Continuously Frame-Convertible Sequences}
\author{Chad Berner}

\begin{document}

\maketitle

\begin{abstract}
Frame theory provides a robust method for recovering vectors in a Hilbert space from inner product data, though the associated decomposition formula can be computationally demanding. We relax the frame condition by studying sequences that can be continuously mapped to Parseval frames, yielding a similar reconstruction formula. We characterize such sequences in terms of their analysis operators, without reference to any continuous mapping. We present examples, including sequences that are not complete and those containing no frame sequence. We also give norm-based criteria for when unconditional Schauder sequences and finite unions of bounded unconditional Schauder sequences admit this property. Finally, we classify finite Borel measures on the torus for which the standard exponential system has this property and forms a Riesz–Fischer sequence.
\end{abstract}

\section{Introduction}

Frame theory, originating from Duffin and Schaeffer \cite{Duffin1952Class}, provides useful and stable tools to recover vectors in a Hilbert space $\Hcal$ from inner product data. Specifically, the frame decomposition formula for a frame $\{f_n\}_{n=0}^{\infty}\subseteq \Hcal$ provides a representation of any vector $f\in \Hcal$ via its inner product data as follows:

$$f=\sum_{n=0}^{\infty}\langle f, f_n\rangle S^{-1}f_n=\sum_{n=0}^{\infty}\langle f, S^{-1}f_n\rangle f_n=\sum_{n=0}^{\infty}\langle S^{-1}f, f_n\rangle f_n$$ where $S\in B(\Hcal)$ denotes the positive and invertible frame operator corresponding to $\{f_n\}$. Moreover, it is known that for a frame $\{f_n\}_{n=0}^{\infty}\subseteq \Hcal$, for all $f\in \Hcal$,
$$f=\sum_{n=0}^{\infty} \left\langle f, S^{-\frac{1}{2}}f_n \right\rangle S^{-\frac{1}{2}}f_n.$$ That is, $\{S^{-\frac{1}{2}}f_n\}$ is a Parseval frame.

However, in practice, computing the inverse (or inverse square root) of the frame operator S can be difficult. Since the frame decomposition formula already can require applying a nontrivial invertible operator to the sequence, it is natural to relax the condition and ask:
which sequences $\{f_n\}\subseteq \Hcal$ have the property that there exists an invertible $B\in B(\Hcal)$ such that $\{Bf_n\}$ is a Parseval frame. Yet, it is easy to see that this is an equivalent definition of a frame. Therefore, in this paper, we instead investigate which sequences $\{f_n\}\subseteq \Hcal$ have the property that there exists a $B\in B(\Hcal)$ such that $\{Bf_n\}$ is a Parseval frame. A sequence with this property will be called a \textbf{continuously frame-convertible sequence} (CFC sequence). We provide a complete characterization of CFC sequences in terms of the analysis operator, without reference to a bounded operator. Our analysis shows that, although CFC sequences share some properties with frames, they differ from them in substantial ways. In particular, we will show that CFC sequences, although they admit frame-like reconstruction properties, need not be Bessel sequences, need not be complete, and may not even contain frame sequences.

There is substantial research on relaxing the frame condition for sequences in a Hilbert space in order to understand the consequences for stable signal recovery. For instance, Li and Ogawa studied frame series and reconstruction in the weak sense \cite{Li2001Pseudo-duals}. Additionally, Balazs, Antoine, Grybos, and others have investigated weighted frames, that is, sequences that can be converted into frames via a sequence of scalars; see \cite{Balazs2010Weighted, Dutkay2016Weighted, Balazs2023Weighted}. Furthermore, Christensen provided results on frame expansions utilizing bounded operators \cite{Christensen1995Frames}. Antoine, Corso, and Trapani also studied metric operators in relation to frames \cite{Antoine2021Lower}, considering when $\{Bf_n\}$ forms a frame for positive definite (possibly unbounded) operators $B$. Their setting is more general in the sense that it allows unbounded positive operators, whereas we restrict attention to bounded operators without a positivity assumption and obtain a more explicit characterization. In \cite{Berner2026Sequences}, sequences with a frame-like decomposition were investigated; this condition is strictly stronger than the CFC property considered here. Finally, the Kaczmarz algorithm offers an independent approach to recovering vectors in a Hilbert space from inner product data; see \cite{Kwapien2001Kaczmarz, Haller2005Kaczmarz}. In this work, we provide examples of weighted frames and sequences arising from the Kaczmarz algorithm, and identify specific cases in which they are, or fail to be, CFC sequences.

The primary aim of this paper is to establish a classification of CFC sequences in terms of the analysis operator, without reference to any bounded operator. A secondary objective is to determine when specific sequences, including finite unions of bounded Schauder sequences and exponential systems, are CFC sequences.

The outline of this paper is as follows. In Section \ref{3}, we classify CFC sequences in terms of the analysis operator and provide examples of CFC sequences that fail to exhibit properties typical of frames. In Section \ref{4}, we show that Riesz–Fischer sequences are CFC sequences and classify them in terms of the analysis operator. In Sections \ref{5}, \ref{6}, we provide a norm-based classification for when unconditional Schauder sequences and finite unions of bounded unconditional Schauder sequences are CFC sequences. Finally, in Section \ref{7}, we classify finite Borel measures on the torus for which the standard exponential system is a CFC sequence and, separately, those for which it is a Riesz–Fischer sequence.

\section{Preliminaries}

We begin by providing the definition of a frame:
\begin{def_n}
Let $\{f_{n}\}_{n}\subseteq \Hcal$ where $\Hcal$ is a Hilbert space.
\begin{enumerate}
\item If there exists $A>0$ such that
$$A\|f\|^{2}\leq \sum_{n}|\langle f, f_{n}\rangle|^{2}$$
for all $f\in \Hcal$, then $\{f_{n}\}_{n}$ is called a \textbf{lower semi-frame}.
\item If there exists $B>0$ such that
$$\sum_{n}|\langle f, f_{n}\rangle|^{2}\leq B\|f\|^{2}$$
for all $f\in \Hcal$, then $\{f_{n}\}_{n}$ is called a \textbf{Bessel sequence}.
\item If there exists $A,B>0$ such that
    $$A\|f\|^{2}\leq \sum_{n}|\langle f, f_{n}\rangle|^{2}\leq B\|f\|^{2}$$
    for all $f\in \Hcal$, then $\{f_{n}\}_{n}$ is called a \textbf{frame}. Furthermore, if $\Hcal=\overline{\operatorname{span}\{f_n\}}$, then $\{f_{n}\}_{n}$ is called a \textbf{frame sequence}.
\item If
$$\|f\|^{2}=\sum_{n}|\langle f, f_{n}\rangle|^{2}$$
for all $f\in \Hcal$, then $\{f_{n}\}_{n}$ is called a \textbf{Parseval frame}.
\item If $\{f_{n}\}$ is a frame, and there is a $\{g_{n}\}_{n}\subseteq \Hcal$ such that
$$\langle f_n,g_k\rangle=\delta_{nk},$$then $\{f_{n}\}$ is called a \textbf{Riesz basis}. Furthermore, if $\Hcal=\overline{\operatorname{span}\{f_n\}}$, then $\{f_{n}\}_{n}$ is called a \textbf{Riesz sequence}.
\end{enumerate}
Moreover, if $\{f_{n}\}_{n}$ is a Bessel sequence in Hilbert space $\Hcal$,
the map on $\Hcal$: $f\to \{\langle f, f_{n}\rangle \}_{n}$ is called the \textbf{analysis operator}, its adjoint is called the \textbf{synthesis operator}, and 
we denote $S: \Hcal\to \Hcal$ as the positive operator such that
$$S(f)=\sum_{n}\langle f, f_{n}\rangle f_{n},$$ which is called the \textbf{frame operator}.
Additionally, if $\{f_{n}\}_{n}$ is a frame, then its associated frame operator is invertible, and $\{S^{-1}f_{n}\}_{n}$ is a frame such that
$$f=\sum_{n}\langle f, S^{-1}f_{n}\rangle f_{n}=\sum_{n}\langle f, f_{n}\rangle S^{-1}f_{n}$$ with convergence in norm for all $f\in \Hcal$.

Additionally, if $\{f_{n}\}$ and $\{g_{n}\}$ are frames (Bessel sequences) such that for all $f\in \Hcal$,
$$f=\sum_{n}\langle f, f_{n}\rangle g_{n},$$ then $\{f_{n}\}$ and $\{g_{n}\}$ are called \textbf{dual frames}. In particular, many frames have multiple dual frames. Additionally, because the synthesis operator and analysis operator associated with both sequences are bounded, by an adjoint calculation, 
$$f=\sum_{n}\langle f,g_{n}\rangle f_{n}$$ also holds for all $f\in \Hcal$. Finally, the norm convergence of frame series is also unconditional.
\end{def_n}

Furthermore, we define Schauder bases:
\begin{def_n}
Let $\{f_{n}\}_{n}\subseteq \Hcal$ where $\Hcal$ is a Hilbert space. $\{f_n\}_{n}$ is called a \textbf{Schauder basis} if there is a $\{g_{n}\}_{n}\subseteq \Hcal$ such that
$$\langle f_n,g_k\rangle=\delta_{nk}$$ and
\begin{equation}\label{schauderdef}
f=\sum_{n}\langle f,g_n\rangle f_n
\end{equation}
with convergence in norm for all $f\in \Hcal$.
If the convergence in equation \eqref{schauderdef} is unconditional for all $f\in \Hcal$, then $\{f_n\}_n$ is called an \textbf{unconditional Schauder basis}. Again, if $\Hcal=\overline{\operatorname{span}\{f_n\}}$, then $\{f_n\}_n$ is called an \textbf{(unconditional) Schauder sequence}.
\end{def_n}
\begin{rem}
A Schauder basis may not be a frame, and a frame may not be a Schauder basis. However, Riesz bases are exactly the sequences that are both frames and Schauder bases. Moreover, Riesz bases are exactly sequences that are similar to an orthonormal basis. Therefore, it follows that any subsequence of a Riesz sequence is a Riesz sequence.
\end{rem}

Additionally, we will make good use of the following results about Riesz bases \cite{Gohberg1969Theory} and the recently confirmed Feichtinger conjecture \cite{Marcus2015Interlacing}:
\begin{thm*}[Gohberg]
Let $\Hcal$ be a Hilbert space. $\{f_n\}\subseteq \Hcal$
is a Riesz basis if and only if it
is an unconditional Schauder basis and
$$0 < \inf_{n}
\|f_n\| \leq \sup_{n}
\|f_n\| < \infty.$$
\end{thm*}

\begin{thm*}[Marcus, Spielman, Srivastava]
Every norm bounded below frame in a Hilbert space can be partitioned into
finitely many Riesz sequences.
\end{thm*}

Throughout this paper, $\Hcal$ denotes a separable infinite dimensional Hilbert space, $P_{D}$ denotes the orthogonal projection onto subspace $D\subseteq \Hcal$, and $A_{f_n}$ denotes the possibly unbounded analysis operator corresponding to $\{f_n\}\subseteq \Hcal$. In particular,
$A_{f_n}: \operatorname{dom}(A_{f_n})\subseteq \Hcal\to \ell^{2}(\mathbb{N})$ where 
$$A_{f_n}(f)=\{\langle f, f_n\rangle\}_{n=0}^{\infty}$$ and
    $$\operatorname{dom}(A_{f_n})=\left\{f\in \Hcal:\{\langle f, f_n\rangle\}_{n=0}^{\infty}\in \ell^{2}(\mathbb{N}) \right\} .$$

\section{Continuously frame convertible sequences}\label{3}

We begin by defining CFC sequences and provide a classification of CFC sequences in terms of the analysis operator, without reference to any bounded operator. We also present two examples of CFC sequences, one of which is not complete and one of which does not contain a frame sequence. This illustrates that the notion of a CFC sequence is significantly weaker than that of a frame. Intuitively, it requires only that ${f_n}$ can be continuously squeezed or stretched into a sequence that admits perfect frame reconstruction.

\begin{def_n}
A sequence $\{f_n\}_{n=0}^{\infty}\subseteq \Hcal$ is a \textbf{continuously frame-convertible sequence} or \textbf{CFC sequence} if there is a $B\in B(\Hcal)$ such that $\{Bf_n\}$ is a Parseval frame.

Furthermore, if $B$ can be chosen to be injective or surjective, we will say $\{f_n\}$ is an \textbf{ICFC sequence} or $\textbf{SCFC sequence}$ respectively.
\end{def_n}

\begin{thm}\label{mainclassthm}
The following are equivalent:
\begin{enumerate}
    \item $\{f_n\}_{n=0}^{\infty}\subseteq \Hcal$ is a CFC sequence
    \item $\operatorname{Im}(A_{f_n})$ contains an infinite dimensional closed subspace.
    \item There is an infinite-dimensional subspace $D\subseteq \Hcal$ and $A>0$ such that
$$A\|f\|^2 \leq \sum_{n=0}^{\infty}|\langle f, f_n\rangle|^{2}<\infty $$
for all $f\in D$.
    
\end{enumerate}

Furthermore, assuming $(3)$,
\begin{enumerate}
    \item If $D$ is closed, $\{f_n\}$ is a SCFC sequence. Conversely, if $\{f_n\}$ is a SCFC sequence, $D$ may be chosen to be closed.
    \item If $\overline{D}=\Hcal$, $\{f_n\}$ is an ICFC sequence. Conversely, if $\{f_n\}$ is an ICFC sequence, $D$ may be chosen so that $\overline{D}=\mathcal{H}.$
    \item 
   $\left\{(TT^{*})^{\frac{1}{2}}f_n  \right\}\subseteq \overline{D}$ is a Parseval frame in $\overline{D}$ where $T\in B(C, \Hcal)$ and $C$ is a closed subspace of $\ell^{2}(\mathbb{N})$.
\end{enumerate}

\end{thm}
\begin{proof}
$(1)\implies (2)$

It follows that for all $f\in \Hcal$,
\begin{equation}\label{analysisopeq}
\frac{1}{\|B\|^{2}} \|B^{*}f\|^{2}\leq \|f \|^{2}=\sum_{n=0}^{\infty}|\langle B^{*}f, f_n\rangle|^{2}<\infty.  
\end{equation}

Therefore if $\{A_{f_n}(B^{*}g_k)\}$ is a Cauchy sequence, then $\{g_k\}$ is a Cauchy sequence. It follows by continuity of $B^{*}$ that $B^{*}g_k\to B^{*}g$ for some $g\in \Hcal$. As a result, it must be that $\{\langle B^{*}g_k, f_n\rangle\}_{n=0}^{\infty}\to \{\langle B^{*}g, f_n\rangle\}_{n=0}^{\infty}$ in $\ell^2(\mathbb{N})$ by uniqueness of limits. Now we have shown that $A_{f_n}|_{\operatorname{Im}(B^*)}$ is defined, injective, and has closed range. Furthermore, $\operatorname{Im}(A_{f_n}|_{\operatorname{Im}(B^*)})$ is infinite dimensional since it is clear that $\operatorname{Im}(B^{*})$ is infinite dimensional.

$(2)\implies (3)$

Let $C$ denote the infinite dimensional closed subspace contained in \newline $\operatorname{Im}(A_{f_n})$. Also define 
$$D:=\left\{f\in \overline{\operatorname{span}\{f_n\}}: A_{f_n}\in C \right\}.$$
Now note $A_{f_n}|_{D}$ is invertible onto $C$. To show that this operator is also closed, suppose $\{g_k\}\subseteq D$ and $\{A_{f_n}|_{D}(g_k)\}$ are Cauchy sequences. It follows since $A_{f_n}|_{D}$ is invertible and $C$ is closed that $A_{f_n}|_{D}(g_k)\to A_{f_n}|_{D}(f)$ for some $f\in D$. Now since $A_{f_n}|_{D}(g_k)\to A_{f_n}|_{D}(g)$ pointwise as well where $g$ is such that
$g_k\to g$, it follows that 
$f=g$ since $f,g\in \overline{\operatorname{span}\{f_n\}}$.
Therefore, since we have shown that $A_{f_n}|_{D}$ is a closed invertible operator with closed range, its inverse is continuous by the closed graph theorem and $(3)$ follows.

$(3)\implies (1)$

It follows that $(A_{f_n}|_{D})^{-1}$ whose codomain is extended to $\Hcal$, is bounded. We will denote its bounded extension by $A_{f_n}^{-1}$. Therefore, for all $f\in D,$
$$\|A_{f_n}A^{-1}_{f_n}\{\langle f, f_n\rangle\}\|^{2}=\|\{\langle f, f_n\rangle\}\|^{2}=\sum_{n=0}^{\infty}|\left\langle \{\langle f, f_n \right\rangle\},(A^{-1}_{f_n})^{*}f_n\rangle|^{2},$$ showing $\left\{(A^{-1}_{f_n})^{*}f_n \right\}\subseteq \overline{\operatorname{Im}(A_{f_n}|_{D})}$ is a Parseval frame in $\overline{\operatorname{Im}(A_{f_n}|_{D})}$. Therefore, $\left\{U(A^{-1}_{f_n})^{*}f_n \right\}\subseteq \Hcal$ is a Parseval frame in $\Hcal$ where $U: \overline{\operatorname{Im}(A_{f_n}|_{D})}\to \Hcal$ is any unitary.

Now assuming $(3)$, by the proof of $(3)\implies (1)$, if $D$ is closed, $(A^{-1}_{f_n})^{*}$ has closed range and therefore, $U(A^{-1}_{f_n})^{*}$ is surjective. Conversely if $\{Bf_n\}$ is a Parseval frame for some bounded and surjective $B$, then $\operatorname{Im}B^{*}$ is closed, which we can choose to be $D$ as in the proof of $(1)\implies (2)$.

Furthermore, if $\overline{D}=\Hcal$, 
by the proof of $(3)\implies (1)$, $(A_{f_n}^{-1})^{*}$ would be injective and hence $U(A^{-1}_{f_n})^{*}$ is injective. Conversely if $\{Bf_n\}$ is a Parseval frame for some bounded and injective $B$, then $\overline{\operatorname{Im}B^{*}}=\Hcal$, which we can choose to be $D$ as in the proof of $(1)\implies (2)$.

Finally, note that there is a unitary $U: \overline{Im(A_{f_n}|_{D})}\to \overline{D}$ where $$U(A_{f_n}^{-1})^{*}=[A_{f_n}^{-1}(A_{f_n}^{-1})^{*}]^{\frac{1}{2}}.$$
\end{proof}

\begin{rem}
It is clear that if there is a $B\in B(\Hcal)$ such that $\{Bf_n\}$ is a Parseval frame, then $B$ is not unique since the image of a Parseval frame under any co-isometry is also a Parseval frame.
\end{rem}

\begin{rem}
By the uniform boundedness principle, a SCFC sequence is the same as a sequence such that there is a infinite dimensional closed subspace $D\subseteq \Hcal$ where $\{P_{D}f_n\}$ is a frame in $D$.
\end{rem}

\begin{cor}
Let $\{f_n\}_{n=0}^{\infty}\subseteq \Hcal$ be a Bessel sequence. Then $\{f_n\}$ is a CFC sequence if and only if there is a infinite-dimensional subspace $D\subseteq \Hcal$ and $A>0$ such that $$A\|f\|^{2}\leq \sum_{n=0}^{\infty}|\langle f, f_n\rangle|^{2}$$ for all $f\in D.$

Furthermore,

\begin{enumerate}
    \item If $\{f_n\}$ is a CFC sequence, then $\{f_n\}$ is a SCFC sequence.
    \item $\{f_n\}$ is an ICFC sequence if and only if $\{f_n\}$ is a frame.
\end{enumerate} 
\end{cor}

\begin{cor}
 Let $\{f_n\}_{n=0}^{\infty}\subseteq \Hcal$ be a lower-semi frame. Then $\{f_n\}$ is a CFC sequence if and and only if there is a infinite-dimensional subspace $D\subseteq \Hcal$ such that $$\sum_{n=0}^{\infty}|\langle f, f_n\rangle|^{2}<\infty$$ for all $f\in D.$
 
 \begin{enumerate}
 \item $\{f_n\}$ is a SCFC sequence if and only if there is an infinite dimensional closed subspace $D$ such that $\{P_{D}f_n\}$ is a Bessel sequence.
     \item $\{f_n\}$ is an ICFC sequence if and only if there is a dense subspace $D\subseteq \Hcal$ such that $$\sum_{n=0}^{\infty}|\langle f, f_n\rangle|^{2}<\infty$$ for all $f\in D.$
 \end{enumerate}  
\end{cor}
We will see soon that there exists lower-semi frames that are not CFC sequences.

Another key property of frames is that the class of Bessel sequences is exactly the class of sequences that can be extended to frames. Similarly, it is not hard to check using Theorem \ref{mainclassthm} there is a similar result for CFC sequences:

\begin{cor}\label{extcor}
A sequence $\{f_n\}_{n=0}^{\infty}\subseteq \Hcal$ can be extended to a CFC sequence if and only if $\operatorname{dom}(A_{f_n})$ contains an infinite dimensional subspace.
\end{cor}

One of the main differences between frames and CFC sequences is that even ICFC sequences don't need to be complete:
\begin{thm}
There exists an incomplete ICFC sequence $\{f_n\}_{n=0}^{\infty}\subseteq \ell^{2}(\mathbb{N})$.
\end{thm}
\begin{proof}
Let $\{e_n\}_{n=0}^{\infty}\subseteq \ell^{2}(\mathbb{N})$ denote the standard orthonormal basis. Define $f_0=0$ and $f_n=ne_{n}$ for $n\geq 1$. Clearly $\{f_n\}$ is not complete. 

However, define 
$$D:=\left\{ \left\{a(k) \right\}_{k=0}^{\infty}\in c_{00}:\sum_{k=0}^{\infty}a(k)=0 \right\}$$where $c_{00}$ denotes sequences in $\ell^{2}(\mathbb{N})$ with finite support. Let $j\geq0,r>0$ be arbitrary and define $a^{j}(k)\in D$ by $a^{j}(j)=1$ and $a^{j}(k)=-\frac{1}{r}$ for any $r$ choices of $k\neq j$. It is easy to see that $\|e_{j}-a^{j}(k)\|^{2}=\frac{1}{r}$. Thus, $\overline{D}=\ell^{2}
(\mathbb{N})$. 

Now consider for any $a(k)\in D$, 
$$\sum_{k=0}^{\infty}|a(k)|^{2}=\left|\sum_{k=1}^{\infty}-a(k) \right|^{2}+\sum_{k=1}^{\infty}|a(k)|^{2}\leq $$
$$\left(\sum_{k=1}^{\infty}\frac{1}{k^{2}} \right)\left(\sum_{k=1}^{\infty}k^{2}|a(k)|^{2} \right)+\sum_{k=1}^{\infty}|a(k)|^{2}\leq \left(\frac{\pi^{2}}{6}+1 \right)\left(\sum_{k=1}^{\infty}k^{2}|a(k)|^{2} \right)<\infty.$$
It follows that $\{f_n\}$ is a ICFC sequence by Theorem \ref{mainclassthm}.
\end{proof}
Another property that differs between frames and CFC sequences is that a frame that is also a basis has the property that every subsequence is a frame sequence. The following example shows that even unconditional Schauder bases that are also SCFC sequences can have all subsequences fail to be frame sequences.
\begin{thm}
There exists an unconditional Schauder basis $\{f_n\}_{n=0}^{\infty}\subseteq \ell^{2} (\mathbb{N})$ that is a SCFC sequence, but no subsequence of $\{f_n\}_{n=0}^{\infty}$ is a frame sequence.
\end{thm}

\begin{proof}
Consider the standard orthonormal basis $\{e_n\}_{n=0}^{\infty}\subseteq \ell^{2}(\mathbb{N})$. Now define $\{g_n\}_{n=0}^{\infty}$ to be the following Parseval frame:
$$\left\{e_0, e_1, \frac{e_2}{\sqrt{2}},\frac{e_2}{\sqrt{2}}, \frac{e_3}{\sqrt{3}},\frac{e_3}{\sqrt{3}},\frac{e_3}{\sqrt{3}},\dots \right\}.$$ Let $D=\operatorname{Im}(A_{g_n})$, and define $\{a_n\}_{n=0}^{\infty}$ to be the following sequence of positive integers:
$$\left\{1, 1, \sqrt{2}, \frac{1}{\sqrt{2}}, \sqrt{3}, \frac{1}{\sqrt{3}},\frac{1}{\sqrt{3}},\dots \right\}.$$ It is easy to check that $\{a_ng_n\}_{n=0}^{\infty}$ is a Bessel sequence with bound $2$.

We claim that $\{P_{D}a_ne_n\}\subseteq D$ is a frame in $D$, and no subsequence of $\{a_ne_n\}_{n=0}^{\infty}$ is a frame sequence. First note that for $f\in D$,
$$\|f\|^2\leq \sum_{n=0}^{\infty}|\langle f, a_ne_n\rangle|^{2}=\sum_{n=0}^{\infty}|a_n|^{2}|\langle g, g_n\rangle|^{2}\leq 2\|g\|^{2} =2\|f\|^{2}$$ for some $g\in \ell^{2}(\mathbb{N})$.

Now suppose for the sake of contradiction that $\{a_{n_k}e_{n_k}\}_{k=0}^{\infty}$ is a frame sequence for some subsequence. Then $\{a_{n_k}e_{n_k}\}$ must be a Riesz sequence since it is biorthogonal to $\{a_{n_k}^{-1}e_{n_k}\}$. Therefore, $\{a_{n_k}e_{n_k}\}$ must be norm bounded above and below, which is a contradiction.
\end{proof}

\section{Continuously orthonormalizable sequences}\label{4}

In \cite{Young2001Introduction}, Young proved that sequences $\{f_n\}_{n=0}^{\infty}$ whose analysis operators are surjective are exactly sequences with the following property: There exists an $A>0$ such that

$$A \sum_{n=0}^{k}|a_n|^2 \leq \left\| \sum_{n=0}^{k} a_n f_n  \right\|^2$$
for all finite sequences $\{a_n\}_{n=0}^{\infty}$. Sequences that satisfy this inequality are called \textbf{Riesz-Fischer sequences}. There are also results on Riesz-Fischer sequences of exponentials in \cite{Reid1995Class}. For more on Riesz-Fischer sequences, the reader can see \cite{Casazza2002Riesz-Fischer, Young1998Class}.

A natural strengthening of the CFC condition is to require that the transformed sequence $\{Bf_n\}$ is not merely a Parseval frame, but an orthonormal basis. Such sequences are, in a sense, very “Gram–Schmidt-friendly”: there exists a bounded linear operator that reshapes them into an orthonormal basis in one step, without the step-by-step procedure of the classical Gram–Schmidt algorithm. We call sequences with this property \textbf{continuously orthonormalizable (CO).} Similarly, if $B$ can be chosen to be injective or surjective, we will say $\{f_n\}$ is an \textbf{ICO or SCO sequence} respectively. However, it turns out that $CO$ sequences are exactly Riesz-Fischer sequences, which we will show here. This result is know in \cite{Casazza2002Riesz-Fischer}, but our classification result is also in terms on the analysis operator as a result of Theorem \ref{mainclassthm}.

\begin{cor}\label{COcor}
The following are equivalent:
\begin{enumerate}
    \item $\{f_n\}_{n=0}^{\infty}\subseteq \Hcal$ is a CO sequence.
    \item $A_{f_n}$ is surjective.
    \item There is an infinite-dimensional subspace $D\subseteq \Hcal$ and $A>0$ such that
$$A\|f\|^2 \leq \sum_{n=0}^{\infty}|\langle f, f_n\rangle|^{2}<\infty $$
for all $f\in D$, and there is a $\{g_{n}\}\subseteq D$ such that
$$\langle f_n,g_k\rangle=\delta_{nk}.$$
\item $\{f_n\}_{n=0}^{\infty}\subseteq \Hcal$ is a Riesz-Fischer sequence.
\end{enumerate}

Furthermore assuming $(3)$,
\begin{enumerate}
    \item If $D$ is closed, $\{f_n\}_{n=0}^{\infty}\subseteq \Hcal$ is a SCO sequence. Conversely, if $\{f_n\}_{n=0}^{\infty}\subseteq \Hcal$ is a SCO sequence, $D$ may be chosen to be closed.
    \item If $\overline{D}=\Hcal$, $\{f_n\}_{n=0}^{\infty}\subseteq \Hcal$ is an ICO sequence. Conversely, if \newline $\{f_n\}_{n=0}^{\infty}\subseteq \Hcal$ is an ICO sequence, $D$ may be chosen so that $\overline{D}=\mathcal{H}.$
    \item 
   $\left\{(TT^{*})^{\frac{1}{2}}f_n \right\}\subseteq \overline{D}$ is an orthonormal basis in $\overline{D}$ where \newline $T\in B(\ell^2(\mathbb{N}), \Hcal)$.
\end{enumerate}
\end{cor}
\begin{proof}
$(1)\implies (2)$

Following the proof of $(1)\implies (2)$ from Theorem \ref{mainclassthm}, in this case, $\operatorname{Im}(A_{f_n}|_{\operatorname{Im}(B^{*})})$ contains sequences of finite support since 
$$\langle B^{*}Bf_n,f_k\rangle=\delta_{nk}.$$
Statement $(2)$ follows since we have already shown $\operatorname{Im}(A_{f_n}|_{\operatorname{Im}(B^{*}}))$ is closed.

$(2)\implies (3)$

This easily follows from the proof of $(2)\implies (3)$ from Theorem \ref{mainclassthm} and because $$\langle B^{*}Bf_n,f_k\rangle=\delta_{nk}.$$

$(3)\implies (4)$

Following the proof of $(3)\implies (1)$ from Theorem \ref{mainclassthm},
$\left\{(A^{-1}_{f_n})^{*}f_n \right\}$ is a Parseval frame in $\ell^{2}(\mathbb{N})$ since $\{g_n\}\subseteq D$. In particular, the frame operator associated with $\left\{(A^{-1}_{f_n})^{*}f_n \right\}$  is the identity operator.
Furthermore, 
$$\left\langle (A^{-1}_{f_n})^{*}f_j,A_{f_n}g_k\right\rangle= \delta_{jk}.$$
Then it follows from the frame reconstruction property that $\left\{(A^{-1}_{f_n})^{*}f_n \right\}$ is biorthogonal to itself and is therefore an orthonormal basis. In particular, $\{Bf_n\}$ is an orthonormal basis for some $B\in B(H)$. Now observe that
$$\left\|\sum_{n=0}^{k}a_n Bf_n \right \|^2\leq \|B\|^2 \left\|\sum_{n=0}^{k}a_n f_n \right \|^2$$ for all finite sequences $\{a_n\}_{n=0}^{\infty}.$ Now statement $(4)$ follows since $\{Bf_n\}$ is an orthogonal sequence.

$(4)\implies (1)$

Let $\{e_n\}_{n=0}^{\infty}\subseteq \Hcal$ denote any orthonormal basis. It follows from the Riesz-Fischer inequality that the linear map defined by $f_n \to e_n$ on the $\operatorname{span}\{f_n\}$ is bounded.

The other three statements follow very similarly to the proof of Theorem \ref{mainclassthm}.
\end{proof}

\begin{cor}
The following are equivalent:
\begin{enumerate}
    \item $\{f_n\}_{n=0}^{\infty}\subseteq \Hcal$ is a SCO sequence.
    \item There is an infinite dimensional closed subspace $D$ such that \newline $\{P_{D}f_n\}_{n=0}^{\infty}$ is a Riesz basis in $D$.
\end{enumerate}
\end{cor}

We will make good use of the follow result, which is known in \cite{Casazza2002Riesz-Fischer}:

\begin{cor*}[P. Casazza et al.]\label{RFcor}
A Riesz-Fischer sequence $\{f_n\}\subseteq \Hcal$ is complete if and only if it is a lower-semi frame.
\end{cor*}

Finally, we close this section with a Paley-Wiener type result for CFC and CO sequences, which utilizes the results of Christensen in \cite{Christensen1995Paley-Wiener} and Young in \cite{Young1980Nonharmonic}.

\begin{prop}
Suppose that $\{f_n\}_{n=0}^{\infty}\subseteq \Hcal$ is such that $\{Bf_n\}$ is a Parseval frame for some $B\in B(\Hcal)$. Suppose further that $\{f_n-g_n\}$ is a Bessel sequence with bound $\mathcal{B}<\frac{1}{\|B\|^2}$, then $\{g_n\}$ is a CFC sequence.

\begin{enumerate}
    \item Furthermore, if $\{f_n\}$ is a SCFC sequence or an ICFC sequence, then so is $\{g_n\}$.
    \item If in addition $\{f_n\}$ is a CO sequence, SCO sequence, or ICO sequence, so is $\{g_n\}$.
\end{enumerate}

\end{prop}
\begin{proof}
It easily follows that $\{Bf_n-Bg_n\}$ is a Bessel sequence with bound $$\mathcal{B}\|B\|^2<1.$$ Now using the result of \cite{Christensen1995Paley-Wiener}, $\{Bg_n\}$ is a frame. Now the result follows since every frame is similar to a Parseval frame.

Now if in addition $\{Bf_n\}$ was an orthonormal basis, we have
$$\left\|\sum_{n}a_n (Bf_n-Bg_n) \right\|^2 \leq \|B\|^2 \left\|\sum_{n}a_n (f_n-g_n) \right\|^2 \leq \|B\|^2 \mathcal {B} \sum_{n=0}^{\infty}|a_n|^2=$$
$$\|B\|^2 \mathcal {B} \left\|\sum_{n=0}^{\infty}a_n Bf_n \right\|^2$$ for all finite sequences $\{a_n\}_{n=0}^{\infty}.$
It follows from the classical Payley-Wiener result in \cite{Young1980Nonharmonic} that $\{Bg_n\}$ is a Riesz basis, which is similar to an orthonormal basis.
\end{proof}

\section{Unconditional Schauder sequence classification}\label{5}
In this section, we characterize when an unconditional Schauder basis is a CFC sequence in terms of the behavior of its norms.
\begin{thm}
 Let $\{f_n\}_{n=0}^{\infty}\subseteq \Hcal$ be an unconditional Schauder basis. 
 
 \begin{enumerate}
     \item $\{f_n\}_{n=0}^{\infty}$ is not a CFC sequence if and only if $$\lim_{n\to \infty}\|f_n\|=0.$$
     \item The following are equivalent:
     \begin{enumerate}
      \item $\{f_n\}_{n=0}^{\infty}$ is a CO sequence.
     \item $\{f_n\}_{n=0}^{\infty}$ is an ICO sequence.
    \item $\inf \|f_n\|>0$.
     \end{enumerate} 
     \item The following are equivalent:
     \begin{enumerate}
         \item $\{f_n\}_{n=0}^{\infty}$ is a SCO sequence.
         \item $\{f_n\}_{n=0}^{\infty}$ is a Riesz basis.
     \end{enumerate}
     
 \end{enumerate}
\end{thm}
\begin{proof}
Note that $\left\{\frac{f_n}{\|f_n\|} \right\}$ is a Bessel sequence with some bound $\mathcal{B}>0$ by the theorem of Gohberg.

(1):

Suppose $$\lim_{n\to \infty}\|f_n\|=0.$$ Assume for the sake of contradiction using Theorem \ref{mainclassthm}, that there is an infinite-dimensional subspace $D\subseteq \Hcal$ and $A>0$ such that
$$A\|f\|^{2}\leq \sum_{n=0}^{\infty}|\langle f, f_n\rangle|^{2}$$ for all $f\in D.$ Choose $N$ so that $$\|f_{n}\|^{2}<\frac{A}{2\mathcal{B}}$$ for all $n\geq N$.
Now note that there is a non-zero $g\in D\cap span\{f_{0},\dots, f_{N}\}^{\perp}$ by $P_{span\{f_{0},\dots ,f_{N}\}}|_{D}$ having infinite-dimensional kernel.
Now observe that
$$\sum_{n=0}^{\infty}|\langle g,f_n\rangle|^{2}=\sum_{n=N+1}^{\infty}\left|\left\langle g,\frac{f_n}{\|f_n\|} \right\rangle \right|^{2}\|f_n\|^{2}\leq \frac{A}{2\mathcal{B}} \sum_{n=0}^{\infty}\left|\left\langle g,\frac{f_n}{\|f_n\|} \right\rangle \right|^{2}\leq \frac{A}{2} \|g\|^{2},$$ which is a contradiction.

Now suppose that for infinitely many $n$, $$\|f_n\|\geq \epsilon$$ for some $\epsilon>0$. Define $B\in B(\Hcal)$ by $Bf_n=\frac{f_n}{\|f_n\|}$ if n is such that $\|f_n\|\geq \epsilon$ and $Bf_n=0$ for other $n$. It is easy to check that $B$ is bounded since $\left\{\frac{f_n}{\|f_n\|} \right\}$ is similar to an orthonormal basis. Furthermore, $\{Bf_n\}$ is an infinite Riesz sequence with some possible zero vectors, which clearly satisfies the hypothesis of Theorem \ref{mainclassthm}.

(2):

$(a)\implies (b):$ Because $$\langle B^{*}Bf_n,f_k\rangle=\delta_{nk}$$ and since $\{f_n\}$ is complete Schauder basis, it follows that for all $f\in  \Hcal$,
$$f=\sum_{n=0}^{\infty}\langle f, B^{*}Bf_n\rangle f_n.$$ In particular, $\overline{Im(B^{*})}=\Hcal$.

$(b)\implies (c):$ It is clear that a Riesz-Fischer sequence is norm bounded below.

$(c)\implies (a)$.
The map $B\in B(\Hcal)$ defined by $Bf_n=\frac{f_n}{\|f_n\|}$ for all $n$ is bounded since $\left\{\frac{f_n}{\|f_n\|} \right\}$ is similar to an orthonormal basis. 

(3):

$(a)\implies (b):$ By the previous $(a)\implies (b)$ argument, $B$ must be injective as well. Therefore, $\{f_n\}$ is similar to an orthonormal basis.

$(b)\implies (a):$ This is clear. 

\end{proof}

\begin{rem}
The previous result can also be restated easily if  $\{f_n\}_{n=0}^{\infty}$ is an unconditional Schauder basis in $\overline{\operatorname{span}\{f_n\}}$.
\end{rem}

The following result shows that weighted frames need not be CFC sequences. It is also clear that a CFC sequence need not be a weighted frame. 
\begin{cor}
Let $\{e_n\}_{n=0}^{\infty}\subseteq \ell^{2}(\mathbb{N})$ be the standard basis.
\begin{enumerate}
    \item $\left\{(n+1)e_{n} \right\}_{n=0}^{\infty}$ is an ICO sequence and not a SCO sequence.
    \item  $\left\{(n+1)^{-1}e_{n} \right \}_{n=0}^{\infty}$ is not a CFC sequence.
    \item $\left\{(n+1)e_{n} \right\}_{n=0}^{\infty}\cup \left\{(n+1)^{-1}e_{n} \right \}_{n=0}^{\infty}$ is a CFC sequence and not a CO sequence.
\end{enumerate}
\end{cor}

\section{Bounded Bessel normalizable sequence classification}\label{6}
Similarly, in this section, we characterize when a bounded Bessel normalizable sequence is a CFC sequence in terms of the behavior of its norms. Bessel normalizable sequences have a rich analysis the reader can see here, \cite{Yu2024Frame-normalizable}.

\begin{thm}
 Let $\{f_n\}_{n=0}^{\infty}\subseteq \Hcal$ be a Bessel sequence consisting of no zero vectors such that $\left\{\frac{f_n}{\|f_n\| } \right\}$ is also a Bessel sequence with bound $\mathcal{B}$. Then $\{f_n\}_{n=0}^{\infty}$ is not a CFC sequence if and only if $$\lim_{n\to \infty}\|f_n\|=0.$$
\end{thm}
\begin{proof}
Let $$\lim_{n\to \infty}\|f_n\|=0.$$ Suppose for the sake of contradiction using Theorem \ref{mainclassthm} that there is an infinite-dimensional subspace $D\subseteq \Hcal$ and $A>0$ such that
$$A\|f\|^{2}\leq \sum_{n=0}^{\infty}|\langle f, f_n\rangle|^{2}$$ for all $f\in D.$ Choose $N$ so that $\|f_{n}\|^{2}<\frac{A}{2\mathcal{B}}$ for all $n\geq N$.
Now note that there is a non-zero $g\in D\cap span\{f_{0},\dots f_{N}\}^{\perp}$ by $P_{span\{f_{0},\dots ,f_{N}\}}|_{D}$ having infinite-dimensional kernel.
Now observe that
$$\sum_{n=0}^{\infty}|\langle g,f_n\rangle|^{2}=\sum_{n=N+1}^{\infty}\left|\left\langle g,\frac{f_n}{\|f_n\|} \right\rangle \right|^{2}\|f_n\|^{2}\leq \frac{A}{2\mathcal{B}} \sum_{n=0}^{\infty}\left|\left\langle g,\frac{f_n}{\|f_n\|} \right\rangle \right|^{2}\leq \frac{A}{2} \|g\|^{2},$$ which is a contradiction.

Now suppose that for infinitely many $n$, $$\|f_n\|\geq \epsilon$$ for some $\epsilon>0$.
By the Feichtinger conjecture and since any subsequence of a Riesz sequence is a Riesz sequence, $\left\{\frac{f_n}{\|f_n\|} \right\}$ consists of a finite union of Riesz sequences. Now it follows by pigeonhole principle that there is a subsequence $\left\{\frac{f_{n_{k}}}{\|f_{n_k}\| } \right\}_{k=0}^{\infty}$ that is a Riesz sequence such that $\inf\|f_{n_k}\|>0.$ Now define the map $B\in B(\Hcal)$ such that  $Bf_{n_k}= \frac{f_{n_{k}}}{\|f_{n_k}\| }$ for all $k$ and
$Bf=0$ for
$f\in \{f_{n_{k}}\}^{\perp}$. It follows that $B$ is bounded. Furthermore, $\{Bf_n\}$ consists of an infinite Riesz sequence and a Bessel sequence since it is clear that the bounded image of a Bessel sequence is a Bessel sequence. It follows from Theorem \ref{mainclassthm} that $\{Bf_n\}$ is a CFC sequence and therefore $\{f_n\}$ is a CFC sequence.
\end{proof}

\begin{rem}
This is the classification of sequences consisting of finite unions of bounded unconditional Schauder sequences by the Feichtinger conjecture.
\end{rem}

\begin{cor}
Any norm bounded below Bessel sequence is a SCFC sequence.
\end{cor}

\begin{cor}
If $\mu$ is a Borel probability measure on $[0,1)$ such that $\{f_n\}\subseteq L^{2}(\mu)$ is a Bessel sequence consisting of exponential functions, then $\{f_n\}$ is a SCFC sequence.
\end{cor}

For more about Bessel sequences of exponetials the reader can see: \cite{Dutkay2011Bessel,He2013Exponential, Nir2018Fourier}.

\section{CFC sequences and the standard exponential system}\label{7}
Finally, in this section, we classify finite Borel measures on $[0,1)$ for which the standard exponential system is a CFC sequence and for which it is a Riesz–Fischer sequence. We further show that for singular measures, a certain frame-like exponential system is never a CFC sequence, even though it is a lower-semi frame.

The following theorem in \cite{Berner2024Frame-like} will help with the classification, and its proof was inspired from the work of Lai \cite{Lai2011Fourier}:
\begin{thm*}[Berner]
Let $\mu$ be a finite Borel measure on $[0,1)$, and let $g$ denote the Radon–Nikodym derivative of its absolutely continuous part with respect to Lebesgue measure. Suppose that there is a Bessel sequence $\{g_{n}\}$ with bound $B$ such that
$$\lim_{M\to \infty}\sum_{n=-M}^{M}\langle f, g_{n}\rangle e^{2\pi i nx}=f$$
for all $f\in L^{2}(\mu)$ with convergence in norm, then $$g(x)>\frac{1}{3B}$$ Lebesgue almost everywhere on its support.
\end{thm*}

\begin{thm}\label{mainFourierthm}
Let $\mu$ be a Borel probability measure on $[0,1)$ that is not singular. Then $\{e^{2\pi i nx}\}_{n\in \mathbb{Z}}\subseteq L^{2}(\mu)$ is a SCFC sequence. 

Conversely, if $\mu$ is a singular Borel probability measure on $[0,1)$ and $B$ is a linear map defined on $span\{e^{2\pi i nx}\}_{n\in \mathbb{Z}}$ such that $\{Be^{2\pi i nx}\}_{n\in \mathbb{Z}}$ is a Parseval frame, then $B$ is not closable.
\end{thm}
\begin{proof}
Suppose $\mu$ is a Borel probability measure on $[0,1)$ that is not singular. Then there is a Borel set of positive Lebesgue measure $S$ and $m\in \mathbb{N}$ where $$\frac{1}{m}\leq g(x)\leq m$$ for $x\in S$ where $g$ denotes the absolutely continuous part of $\mu$. Now let $A$ be a Borel set of positive Lebesgue measure zero such that 
$$\mu_{s}(A^{c})=\int_{A}gdx=0$$ where $\mu_{s}$ denotes the singular part of $\mu$. Furthermore, consider infinite-dimensional subspace
$$D:=span\left\{\chi_{S\cap A^{c}}f: f\in L^{2}(\mu)\right\}\subseteq L^{2}(\mu).$$ Observe for any $f\in D$,
$$\frac{1}{m}\left\|f\right\|^{2}=\frac{1}{m}\int_{S\cap A^{c}}|f|^{2}gdx \leq \int_{S\cap A^{c}}|f|^{2}g^{2}dx=\sum_{n\in \mathbb{Z}} \left|\left\langle f, e^{2\pi i nx}\right \rangle \right|^{2}\leq m \|f\|^{2}.$$ Let $U:\overline{D}\to L^{2}(\mu)$ be a unitary operator. It follows that for all $f\in L^{2}(\mu)$,

$$\frac{1}{m}\|f\|^{2}=\frac{1}{m}\|U^{*}f\|^{2}\leq \sum_{n\in \mathbb{Z}}|\langle U^{*}f, e^{2\pi inx}\rangle|^{2}\leq m\|U^{*}f\|^{2}=m\|f\|^{2}.$$
Therefore, $\{UP_{\overline{D}}e^{2\pi i nx}\}_{n\in \mathbb{Z}}$ is a frame. It is also clear that $UP_{\overline{D}}\in B(L^{2}(\mu))$ is surjective.

Now suppose $\mu$ is singular and there is a linear and closable map $B$ defined on $span\{e^{2\pi i nx}\}_{n\in \mathbb{Z}}$ so that $\{Be^{2\pi i nx}\}_{n\in \mathbb{Z}}$ is a Parseval frame. It follows that $B^{*}$ is densely defined on some set $\mathcal{D}$. Then for all $f\in \mathcal{D}$,
$$\sum_{n\in \mathbb{Z}}|\langle B^{*}f, e^{2\pi i nx}\rangle |^{2}<\infty.$$ In particular, the complex measure $B^{*}f\mu$ is absolutely continuous, and since $\mu$ is singular, it must be that $B^{*}f=0$, which contradicts $\mathcal{D}$ being dense.
\end{proof}

\begin{cor}
Let $\mu$ be a Borel probability measure on $[0,1)$. There is an infinite dimensional closed subspace $C\subseteq L^2(\mu)$ such that $\{P_{C}e^{2\pi i nx}\}_{n\in \mathbb{Z}}$ is a frame in $C$ if and only if $\mu$ is not singular.
\end{cor}

For singular measures however, the following result was shown by Herr and Weber \cite{Herr2017Fourier}, utilizing the Kaczmarz algorithm and the results of Kwapien and Mycielski \cite{Kwapien2001Kaczmarz}:

\begin{thm*}[Herr and Weber]
Let $\mu$ be a singular Borel probability measure on $[0,1)$. Then there exists a Parseval frame $\{g_{n}\}_{n=0}^{\infty}\subseteq L^{2}(\mu)$ such that
$$f=\sum_{n=0}^{\infty}\langle f, g_{n}\rangle e^{2\pi i nx}$$ for all $f\in L^{2}(\mu)$.
\end{thm*}

We show that $\{e^{2\pi i nx}\}_{n=0}^{\infty}\subseteq L^{2}(\mu)$ is never a CFC sequence for any singular measure $\mu$, even though it reproduces all vectors with a dual-like Parseval frame. It is also not hard to check by the previous theorem that $\{e^{2\pi i nx}\}_{n=0}^{\infty}\subseteq L^{2}(\mu)$ is a lower-semi frame whenever $\mu$ is singular.

\begin{thm}
If $\mu$ is a singular Borel probability measure on $[0,1)$, then $\{e^{2\pi i nx}\}_{n=0}^{\infty}\subseteq L^{2}(\mu)$ can't be extended to a CFC sequence.
\end{thm}
\begin{proof}
Suppose that there is $f\in L^{2}(\mu)$ such that 
$$\sum_{n=0}^{\infty}|\langle f, e^{2\pi i nx}\rangle|^2<\infty.$$

Now define a complex measure $v$ so that 
$\hat{v}(n)=\langle f, e^{2\pi i nx}\rangle$ for $n\geq 0$ and 
$\hat{v}(n)=0$ for $n<0$. It follows that $v$ that is absolutely continuous with Radon-Nikodym derivative in $L^{2}([0,1))$.

Furthermore, the complex measure $f\mu -v$ is absolutely continuous by F. and M. Riesz theorem. Then it must be that $f\mu$ is absolutely continuous and $f=0$ since $\mu$ is singular. Therefore the $\operatorname{dom}(A_{e^{2\pi inx}})$ is trivial and therefore, the result follows by Corollary \ref{extcor}.
\end{proof}

\begin{thm}
Let $\mu$ be a Borel probability measure on $[0,1)$ with absolutely continuous part denoted $g$. Then 
\begin{enumerate}
    \item $\{e^{2\pi i nx}\}_{n\in \mathbb{Z}}\subseteq L^2(\mu)$ is a Riesz Fischer sequence if and only if $g$ is bounded below.
    \item $\{e^{2\pi i nx}\}_{n\in \mathbb{Z}}\subseteq L^2(\mu)$ is a SCO sequence if and only if $g$ is bounded above and below.
   \item $\{e^{2\pi i nx}\}_{n\in \mathbb{Z}}\subseteq L^2(\mu)$ is an ICO sequence if and only if $\mu$ is absolutely continuous and is $g$ bounded below.
\end{enumerate}

\end{thm}
\begin{proof}
First we assume that $\{e^{2\pi i nx}\}_{n\in \mathbb{Z}}\subseteq L^2(\mu)$ is a Riesz Fischer sequence. It follows from Corollary \ref{RFcor} that $\{e^{2\pi i nx}\}_{n\in \mathbb{Z}}$ is a lower-semi frame. Furthermore, by Corollary \ref{COcor}, there is a $B\in B(L^2(\mu))$ such that 

\begin{equation}\label{biortheq}
\langle B^{*}Be^{2\pi inx},e^{2\pi ikx}\rangle=\delta_{nk}.\end{equation}
It follows that for each $n$, the complex measure $B^{*}Be^{2\pi i nx}d\mu$ has the same Fourier coefficients as the complex measure $e^{2\pi i nx} dx$. In particular, 
$$B^{*}Be^{2\pi inx}g(x)=e^{2\pi inx}$$ for Lebesgue almost every $x$, and therefore,
$g$ is non vanishing. Consequently, $\frac{1}{g}\in L^{\infty}([0,1))$.

Now suppose that $\{e^{2\pi i nx}\}_{n\in \mathbb{Z}}\subseteq L^2(\mu)$ is a SCFC sequence in addition. Since $\{e^{2\pi i nx}\}_{n\in \mathbb{Z}}\subseteq L^2(\mu)$ is complete and $\frac{1}{g}\in L^{\infty}([0,1))$, it follows from equation \ref{biortheq},
\begin{equation}\label{B*B}
B^{*}Be^{2\pi inx}=\chi_{S}\dfrac{e^{2\pi i nx}}{g}   
\end{equation}
where $S$ is any set of Lebesgue measure one such that $$\mu_{s}(S)=0$$ where $\mu_{s}$ denotes the singular part of $\mu$. Furthermore, since $B^{*}B$ is bounded below from B being surjective and $\{B^{*}Be^{2\pi inx}\}$ is a Bessel sequence, it follows that for all $f$,
$$\sum_{n}|\langle B^{*}Bf, e^{2\pi i nx}\rangle|^2\leq \mathcal{B} \|B^{*}Bf\|^2$$ for some $\mathcal{B}>0$. 
In particular by equation \ref{B*B},
$$\int_{0}^{1}|h|^2dx\leq \mathcal{B}\int_{0}^{1}|h|^2\frac{1}{g}dx $$
for any trig polynomial $h$. It follows from this inequality that $\frac{1}{g}$ is bounded below as trig polynomials are dense in $L^{2}([0,1))$ and because $\frac{1}{g}$ is bounded.

Now suppose $\{e^{2\pi i nx}\}_{n\in \mathbb{Z}}\subseteq L^2(\mu)$ is an ICFC sequence. Note that $B^{*}f \mu$ is an absolutely continuous measure for any $f\in L^2(\mu)$ by the following inequality:
$$\sum_{n=0}^{\infty}|\langle B^{*}f, e^{2\pi i nx}\rangle |^{2}<\infty.$$
Therefore, $B^{*}f=0$ in $L^2(\mu_{s})$ where $\mu_{s}$ denotes the singular part of $\mu$. Now if we assume for the sake of contradiction that there is a Lebesgue measure zero set $A$ where $\mu_{s}(A)>0$, then there is a sequence $B^{*}g_k\to \chi_{A}$ since $B$ is injective. This contradicts $B^{*}g_k=0$ in $L^2(\mu_{s})$, and we conclude that $\mu$ is absolutely continuous.

Now suppose $$g(x)\geq A>0$$ for Lebesgue almost every $x$. Then it follows that the linear map defined by $e^{2\pi i nx}\to \chi_{S}\dfrac{e^{2\pi i nx}}{\sqrt{g}}$ for all $n\in \mathbb{Z}$ is bounded where $S\subseteq [0,1)$ is any set of Lebesgue measure one such that $\mu_{s}(S)=0$ where $\mu_{s}$ denotes the singular part of $\mu$. Furthermore, it is not hard to check that $\left\{\chi_{S}\dfrac{e^{2\pi i nx}}{\sqrt{g}} \right\}_{n\in \mathbb{Z}}\subseteq L^{2}(\mu)$ is an orthonormal sequence. Now it is clear that $\{e^{2\pi i nx}\}_{n\in \mathbb{Z}}\subseteq L^2(\mu)$ is a Riesz-Fischer sequence by Corollary \ref{COcor}.

Now if in addition, $$0<A\leq g(x)\leq B<\infty$$ for Lebesgue almost every $x$. It follows that $\{\chi_{S}e^{2\pi i nx}\}_{n\in \mathbb{Z}}\subseteq L^{2}(\mu)$ is a Riesz sequence where $S\subseteq [0,1)$ is any set of Lebesgue measure one such that $\mu_{s}(S)=0$ where $\mu_{s}$ denotes the singular part of $\mu$. Therefore, $\{e^{2\pi i nx}\}_{n\in \mathbb{Z}}\subseteq L^2(\mu)$ is a SCO sequence by Corollary \ref{COcor}.

Finally, suppose $\mu$ is absolutely continuous and $$g(x)\geq A>0$$ for Lebesgue almost every $x$. In this case, 
$$\left\langle \dfrac{e^{2\pi inx}}{g},e^{2\pi i kx} \right\rangle =\delta_{nk}.$$
Furthermore, $\{e^{2\pi i nx}\}_{n\in \mathbb{Z}}\subseteq L^2(\mu)$ is a lower-semi frame such that \newline $\operatorname{span}\left\{\dfrac{e^{2\pi i nx}}{g}\right\}\subseteq \operatorname{dom}(A_{e^{2\pi i nx}})$. Also, it is clear that $\operatorname{span}\left\{\dfrac{e^{2\pi i nx}}{g}\right\}$ is dense in $L^{2}(\mu)$, showing $\{e^{2\pi inx}\}$ is an ICO sequence by Theorem \ref{mainclassthm}.
\end{proof}

\printbibliography

@book {Young2001Introduction,
    AUTHOR = {Young, Robert M.},
     TITLE = {An introduction to nonharmonic {F}ourier series},
   EDITION = {first},
 PUBLISHER = {Academic Press, Inc., San Diego, CA},
      YEAR = {2001},
     PAGES = {xiv+234},
      ISBN = {0-12-772955-0},
   MRCLASS = {42-01 (30D20 42C15 42C40)},
  MRNUMBER = {1836633},
MRREVIEWER = {Ole\ Christensen},
}

@article {Reid1995Class,
    AUTHOR = {Reid, Russell M.},
     TITLE = {A class of {R}iesz-{F}ischer sequences},
   JOURNAL = {Proc. Amer. Math. Soc.},
  FJOURNAL = {Proceedings of the American Mathematical Society},
    VOLUME = {123},
      YEAR = {1995},
    NUMBER = {3},
     PAGES = {827--829},
      ISSN = {0002-9939,1088-6826},
   MRCLASS = {42A70 (30D99 42C15)},
  MRNUMBER = {1223519},
       DOI = {10.2307/2160807},
       URL = {https://doi.org/10.2307/2160807},
}

@article {Young1998Class,
    AUTHOR = {Young, Robert M.},
     TITLE = {On a class of {R}iesz-{F}ischer sequences},
   JOURNAL = {Proc. Amer. Math. Soc.},
  FJOURNAL = {Proceedings of the American Mathematical Society},
    VOLUME = {126},
      YEAR = {1998},
    NUMBER = {4},
     PAGES = {1139--1142},
      ISSN = {0002-9939,1088-6826},
   MRCLASS = {42A65 (30E05)},
  MRNUMBER = {1452835},
MRREVIEWER = {Kristian\ Seip},
       DOI = {10.1090/S0002-9939-98-04416-5},
       URL = {https://doi.org/10.1090/S0002-9939-98-04416-5},
}

@article {Casazza2002Riesz-Fischer,
    AUTHOR = {Casazza, P. and Christensen, O. and Li, S. and Lindner, A.},
     TITLE = {Riesz-{F}ischer sequences and lower frame bounds},
   JOURNAL = {Z. Anal. Anwendungen},
  FJOURNAL = {Zeitschrift f\"ur Analysis und ihre Anwendungen. Journal for
              Analysis and its Applications},
    VOLUME = {21},
      YEAR = {2002},
    NUMBER = {2},
     PAGES = {305--314},
      ISSN = {0232-2064,1661-4534},
   MRCLASS = {42C15},
  MRNUMBER = {1915263},
MRREVIEWER = {Mohamed\ Sifi},
       DOI = {10.4171/ZAA/1079},
       URL = {https://doi.org/10.4171/ZAA/1079},
}

@article {Antoine2021Lower,
    AUTHOR = {Antoine, J.-P. and Corso, R. and Trapani, C.},
     TITLE = {Lower semi-frames, frames, and metric operators},
   JOURNAL = {Mediterr. J. Math.},
  FJOURNAL = {Mediterranean Journal of Mathematics},
    VOLUME = {18},
      YEAR = {2021},
    NUMBER = {1},
     PAGES = {Paper No. 11, 20},
      ISSN = {1660-5446,1660-5454},
   MRCLASS = {42C15 (46B15 46G10)},
  MRNUMBER = {4193429},
       DOI = {10.1007/s00009-020-01652-x},
       URL = {https://doi.org/10.1007/s00009-020-01652-x},
}

@article {Yu2024Frame-normalizable,
    AUTHOR = {Yu, Pu-Ting},
     TITLE = {Frame-normalizable sequences},
   JOURNAL = {Adv. Comput. Math.},
  FJOURNAL = {Advances in Computational Mathematics},
    VOLUME = {50},
      YEAR = {2024},
    NUMBER = {4},
     PAGES = {Paper No. 89, 23},
      ISSN = {1019-7168,1572-9044},
   MRCLASS = {42C15},
  MRNUMBER = {4784098},
MRREVIEWER = {Tin\ Thien\ Tran},
       DOI = {10.1007/s10444-024-10182-z},
       URL = {https://doi.org/10.1007/s10444-024-10182-z},
}

@article {Li2001Pseudo-duals,
    AUTHOR = {Li, Shidong and Ogawa, Hidemitsu},
     TITLE = {Pseudo-duals of frames with applications},
   JOURNAL = {Appl. Comput. Harmon. Anal.},
  FJOURNAL = {Applied and Computational Harmonic Analysis. Time-Frequency
              and Time-Scale Analysis, Wavelets, Numerical Algorithms, and
              Applications},
    VOLUME = {11},
      YEAR = {2001},
    NUMBER = {2},
     PAGES = {289--304},
      ISSN = {1063-5203,1096-603X},
   MRCLASS = {46B15 (42C15 46C05)},
  MRNUMBER = {1848709},
MRREVIEWER = {Ole\ Christensen},
       DOI = {10.1006/acha.2001.0347},
       URL = {https://doi.org/10.1006/acha.2001.0347},
}

@article {Balazs2010Weighted,
    AUTHOR = {Balazs, Peter and Antoine, Jean-Pierre and Grybo\'s, Anna},
     TITLE = {Weighted and controlled frames: mutual relationship and first
              numerical properties},
   JOURNAL = {Int. J. Wavelets Multiresolut. Inf. Process.},
  FJOURNAL = {International Journal of Wavelets, Multiresolution and
              Information Processing},
    VOLUME = {8},
      YEAR = {2010},
    NUMBER = {1},
     PAGES = {109--132},
      ISSN = {0219-6913,1793-690X},
   MRCLASS = {42C15 (42C40 65T60)},
  MRNUMBER = {2654396},
       DOI = {10.1142/S0219691310003377},
       URL = {https://doi.org/10.1142/S0219691310003377},
}

@article {Dutkay2016Weighted,
    AUTHOR = {Dutkay, Dorin Ervin and Ranasinghe, Rajitha},
     TITLE = {Weighted {F}ourier frames on fractal measures},
   JOURNAL = {J. Math. Anal. Appl.},
  FJOURNAL = {Journal of Mathematical Analysis and Applications},
    VOLUME = {444},
      YEAR = {2016},
    NUMBER = {2},
     PAGES = {1603--1625},
      ISSN = {0022-247X,1096-0813},
   MRCLASS = {42C15},
  MRNUMBER = {3535778},
MRREVIEWER = {Jean-Pierre\ Gabardo},
       DOI = {10.1016/j.jmaa.2016.07.042},
       URL = {https://doi.org/10.1016/j.jmaa.2016.07.042},
}

@misc{Balazs2023Weighted,
      title={Weighted frames, weighted lower semi frames and unconditionally convergent multipliers}, 
      author={Peter Balazs and Rosario Corso and Diana Stoeva},
      year={2023},
      eprint={2310.18957},
      archivePrefix={arXiv},
      primaryClass={math.FA},
      url={https://arxiv.org/abs/2310.18957}, 
}

@article {Christensen1995Frames,
    AUTHOR = {Christensen, Ole},
     TITLE = {Frames and pseudo-inverses},
   JOURNAL = {J. Math. Anal. Appl.},
  FJOURNAL = {Journal of Mathematical Analysis and Applications},
    VOLUME = {195},
      YEAR = {1995},
    NUMBER = {2},
     PAGES = {401--414},
      ISSN = {0022-247X,1096-0813},
   MRCLASS = {47A05},
  MRNUMBER = {1354551},
MRREVIEWER = {I.\ Ya.\ Novikov},
       DOI = {10.1006/jmaa.1995.1363},
       URL = {https://doi.org/10.1006/jmaa.1995.1363},
}

@article {Haller2005Kaczmarz,
    AUTHOR = {Haller, Rainis and Szwarc, Ryszard},
     TITLE = {Kaczmarz algorithm in {H}ilbert space},
   JOURNAL = {Studia Math.},
  FJOURNAL = {Studia Mathematica},
    VOLUME = {169},
      YEAR = {2005},
    NUMBER = {2},
     PAGES = {123--132},
      ISSN = {0039-3223,1730-6337},
   MRCLASS = {41A65 (46C99)},
  MRNUMBER = {2140451},
MRREVIEWER = {Grzegorz\ Lewicki},
       DOI = {10.4064/sm169-2-2},
       URL = {https://doi.org/10.4064/sm169-2-2},
}

@article{Kwapien2001Kaczmarz,
author = {Stanisław Kwapień and Jan Mycielski},
journal = {Studia Mathematica},
language = {eng},
number = {1},
pages = {75-86},
title = {On the Kaczmarz algorithm of approximation in infinite-dimensional spaces},
url = {http://eudml.org/doc/284583},
volume = {148},
year = {2001},
}

@article{Berner2026Sequences,
  author  = {Berner, C.},
  title   = {Sequences that do frame reconstruction},
  journal = {Annals of Functional Analysis},
  year    = {2026},
  volume  = {17},
  pages   = {24},
  doi     = {10.1007/s43034-026-00506-z}
}

@article {Christensen1995Paley-Wiener,
    AUTHOR = {Christensen, Ole},
     TITLE = {A {P}aley-{W}iener theorem for frames},
   JOURNAL = {Proc. Amer. Math. Soc.},
  FJOURNAL = {Proceedings of the American Mathematical Society},
    VOLUME = {123},
      YEAR = {1995},
    NUMBER = {7},
     PAGES = {2199--2201},
      ISSN = {0002-9939,1088-6826},
   MRCLASS = {46C99 (42A99 42C15)},
  MRNUMBER = {1246520},
MRREVIEWER = {B.\ S.\ Rubin},
       DOI = {10.2307/2160957},
       URL = {https://doi.org/10.2307/2160957},
}

@book{Young1980Nonharmonic,
  author    = {Young, Robert M.},
  title     = {An Introduction to Nonharmonic Fourier Series},
  publisher = {Academic Press},
  address   = {New York},
  year      = {1980},
  series    = {Pure and Applied Mathematics},
  volume    = {93},
  isbn      = {9780127728506}
}

@Article{Herr2017Fourier,
AUTHOR = {Herr, John E. and Weber, Eric S.},
TITLE = {Fourier Series for Singular Measures},
JOURNAL = {Axioms},
VOLUME = {6},
YEAR = {2017},
NUMBER = {2},
ARTICLE-NUMBER = {7},
URL = {https://www.mdpi.com/2075-1680/6/2/7},
ISSN = {2075-1680},
DOI = {10.3390/axioms6020007}
}

@book {Gohberg1969Theory,
    AUTHOR = {Gohberg, I. C. and Kre\u in, M. G.},
     TITLE = {Introduction to the theory of linear nonselfadjoint operators},
    SERIES = {Translations of Mathematical Monographs},
    VOLUME = {Vol. 18},
      NOTE = {Translated from the Russian by A. Feinstein},
 PUBLISHER = {American Mathematical Society, Providence, RI},
      YEAR = {1969},
     PAGES = {xv+378},
   MRCLASS = {47.10},
  MRNUMBER = {246142},
}

@article {Marcus2015Interlacing,
    AUTHOR = {Marcus, Adam W. and Spielman, Daniel A. and Srivastava,
              Nikhil},
     TITLE = {Interlacing families {II}: {M}ixed characteristic polynomials
              and the {K}adison-{S}inger problem},
   JOURNAL = {Ann. of Math. (2)},
  FJOURNAL = {Annals of Mathematics. Second Series},
    VOLUME = {182},
      YEAR = {2015},
    NUMBER = {1},
     PAGES = {327--350},
      ISSN = {0003-486X,1939-8980},
   MRCLASS = {46L05 (42A05 46B03 46L30)},
  MRNUMBER = {3374963},
MRREVIEWER = {Robert\ S.\ Doran},
       DOI = {10.4007/annals.2015.182.1.8},
       URL = {https://doi.org/10.4007/annals.2015.182.1.8},
}

@article {Dutkay2011Bessel,
    AUTHOR = {Dutkay, Dorin Ervin and Han, Deguang and Weber, Eric},
     TITLE = {Bessel sequences of exponentials on fractal measures},
   JOURNAL = {J. Funct. Anal.},
  FJOURNAL = {Journal of Functional Analysis},
    VOLUME = {261},
      YEAR = {2011},
    NUMBER = {9},
     PAGES = {2529--2539},
      ISSN = {0022-1236,1096-0783},
   MRCLASS = {42C15 (28A80 46C05 46E30)},
  MRNUMBER = {2826404},
MRREVIEWER = {Ursula\ Maria\ Molter},
       DOI = {10.1016/j.jfa.2011.06.018},
       URL = {https://doi.org/10.1016/j.jfa.2011.06.018},
}

@article {He2013Exponential,
    AUTHOR = {He, Xing-Gang and Lai, Chun-Kit and Lau, Ka-Sing},
     TITLE = {Exponential spectra in {$L^2(\mu)$}},
   JOURNAL = {Appl. Comput. Harmon. Anal.},
  FJOURNAL = {Applied and Computational Harmonic Analysis. Time-Frequency
              and Time-Scale Analysis, Wavelets, Numerical Algorithms, and
              Applications},
    VOLUME = {34},
      YEAR = {2013},
    NUMBER = {3},
     PAGES = {327--338},
      ISSN = {1063-5203,1096-603X},
   MRCLASS = {42C15 (46E30)},
  MRNUMBER = {3027906},
MRREVIEWER = {Keri\ A.\ Kornelson},
       DOI = {10.1016/j.acha.2012.05.003},
       URL = {https://doi.org/10.1016/j.acha.2012.05.003},
}

@article {Duffin1952Class,
    AUTHOR = {Duffin, R. J. and Schaeffer, A. C.},
     TITLE = {A class of nonharmonic {F}ourier series},
   JOURNAL = {Trans. Amer. Math. Soc.},
  FJOURNAL = {Transactions of the American Mathematical Society},
    VOLUME = {72},
      YEAR = {1952},
     PAGES = {341--366},
      ISSN = {0002-9947,1088-6850},
   MRCLASS = {42.4X},
  MRNUMBER = {47179},
MRREVIEWER = {J.\ Korevaar},
       DOI = {10.2307/1990760},
       URL = {https://doi.org/10.2307/1990760},
}

@article {Nir2018Fourier,
    AUTHOR = {Lev, Nir},
     TITLE = {Fourier frames for singular measures and pure type phenomena},
   JOURNAL = {Proc. Amer. Math. Soc.},
  FJOURNAL = {Proceedings of the American Mathematical Society},
    VOLUME = {146},
      YEAR = {2018},
    NUMBER = {7},
     PAGES = {2883--2896},
      ISSN = {0002-9939,1088-6826},
   MRCLASS = {42C15 (42B10)},
  MRNUMBER = {3787351},
MRREVIEWER = {Juan\ Luis\ Varona},
       DOI = {10.1090/proc/13849},
       URL = {https://doi.org/10.1090/proc/13849},
}

@article {Berner2024Frame-like,
    AUTHOR = {Berner, Chad},
     TITLE = {Frame-like {F}ourier expansions for finite {B}orel measures on
              {$\Bbb R$}},
   JOURNAL = {Pure Appl. Funct. Anal.},
  FJOURNAL = {Pure and Applied Functional Analysis},
    VOLUME = {9},
      YEAR = {2024},
    NUMBER = {6},
     PAGES = {1527--1545},
      ISSN = {2189-3756,2189-3764},
   MRCLASS = {42A16 (42A20 42C15 46C07)},
  MRNUMBER = {4853157},
MRREVIEWER = {Beata\ Deregowska},
}

@article {Lai2011Fourier,
    AUTHOR = {Lai, Chun-Kit},
     TITLE = {On {F}ourier frame of absolutely continuous measures},
   JOURNAL = {J. Funct. Anal.},
  FJOURNAL = {Journal of Functional Analysis},
    VOLUME = {261},
      YEAR = {2011},
    NUMBER = {10},
     PAGES = {2877--2889},
      ISSN = {0022-1236,1096-0783},
   MRCLASS = {42C15 (28A80 46B15)},
  MRNUMBER = {2832585},
MRREVIEWER = {Peter\ R.\ Massopust},
       DOI = {10.1016/j.jfa.2011.07.014},
       URL = {https://doi.org/10.1016/j.jfa.2011.07.014},
}

\end{document}